\documentclass[a4paper,11pt]{article}
\usepackage{microtype}
\usepackage[margin=1in]{geometry}
\usepackage{hyperref}
\usepackage{bbold}

\usepackage[T1]{fontenc} 
\usepackage{amsfonts}
\usepackage{amsthm}
\usepackage{authblk}

\usepackage[normalem]{ulem}

\usepackage{xspace}
\usepackage{amsmath,amssymb,enumerate}
\usepackage{yhmath}

\newtheorem{theorem}{Theorem}
\newtheorem{proposition}[theorem]{Proposition}
\newtheorem{lemma}[theorem]{Lemma}
\newtheorem{corollary}[theorem]{Corollary}
\newtheorem{remark}{Remark}

\DeclareMathOperator{\Probability}{\mathbb{P}}

\DeclareMathOperator{\Expected}{\mathbb{E}}

\newcommand{\Ex}[1]{\Expected\pbrcx{#1}}
\newcommand{\Prob}[1]{\Probability\pbrcx{#1}}

\newcommand{\Z}{{\mathbb{Z}}}
\newcommand{\N}{{\mathbb{N}}}
\newcommand{\C}{{\mathbb{C}}}

\newcommand{\pth}[1]{\ensuremath{\left(#1\right)}}
\newcommand{\pbrcx}[1]{\ensuremath{\left[#1\right]}}

\newcommand{\Var}{{\rm var}}

\newcommand{\calB}{{\mathcal{B}}}
\newcommand{\calC}{{\mathcal{C}}}
\newcommand{\calD}{{\mathcal{D}}}

\newcommand{\bfa}{{\bf a}}

\usepackage[mathlines]{lineno}

\title{On Limit Constants for Last Passage Percolation in Transitive Tournaments\footnote{Part of this work was done when the author was supported by the 
Advanced Grant of the European Research Council: GUDHI (Geometry Understanding in Higher Dimensions), Grant agreement no. 339025.}}
\author[1]{Kunal Dutta}

\affil[1]{Department of Mathematics, Informatics and Mechanics, University of Warsaw, Poland\\ K.dutta@mimuw.edu.pl}

\begin{document}

\maketitle


\begin{abstract}
   We investigate the \emph{last passage percolation} problem on transitive tournaments,
   in the case when the edge weights are independent Bernoulli random variables.
   Given a transitive tournament on $n$ nodes with random weights on its edges, 
   the last passage percolation problem seeks to find the weight $X_n$ of the heaviest  
   path, where the weight of a path is the sum of the weights on its edges. 

We give a recurrence relation and use it to obtain a 
(bivariate) generating function for the probability generating function of $X_n$. This also gives exact combinatorial 
expressions for $\Ex{X_n}$, which was stated as an open problem by Yuster [\emph{Disc. Appl. Math.}, 2017]. 
We further determine scaling constants in the limit laws for $X_n$.  
Define $\beta_{tr}(p) := \lim_{n\to \infty} \frac{\Ex{X_n}}{n-1}$. Using singularity analysis, 
we show 
\[ \beta_{tr}(p) = \pth{\sum_{n\geq 1}(1-p)^{{n\choose 2}}}^{-1}. \]
In particular, $\beta_{tr}(0.5) = \pth{\sum_{n\geq 1} 2^{-{n\choose 2}}}^{-1} = 0.60914971106...$.
This settles the question of determining the value of $\beta_{tr}(0.5)$, initiated by Yuster. 
$\beta_{tr}(p)$ is also the limiting value in the strong law of large numbers for $X_n$, given by Foss, 
Martin, and Schmidt [\emph{Ann. Appl. Probab.}, 2014]. We also derive the scaling constants in 
the functional central limit theorem for $X_n$ proved by Foss et al.




\end{abstract}

\section{Introduction}
\label{sec:intro}

The problem of finding the longest or heaviest path in a graph is a classic problem 
in computer science, and also has significance in other areas like operations research, economics, ecological studies, 
etc.~\cite{Cormen2001introduction,cfcb8e03d96a4bfbb5bd9eff03d36585,10.2307/30035650}.
Accordingly it has been extensively studied both from the mathematical and algorithmic points of view.
For general graphs, even deciding the existence of a Hamiltonian path is known to be NP-complete~\cite{10.5555/578533}; 
however for directed acyclic graphs it is solvable in linear time~\cite{DBLP:books/daglib/0037819}. The directed acyclic setting assumes significance because 
several problems like the critical path method~\cite{DBLP:books/daglib/0037819}, layered graph drawing~\cite{10.5555/376944.376949}, etc. 
require the computation of longest paths in such graphs. 
We investigate a probabilistic version of the problem, where we have a directed acyclic graph and the edges have random and independently 
chosen weights. 
In this setting, the problem is an instance of the well-known \emph{last passage percolation problem}, studied in statistical physics 
and probability theory~\cite{Calder2015,Mart2006}. 
Unlike the classical setting for last passage percolation on the $d$-dimensional integer lattice $\Z^d$~\cite{Mart2006,HambMart2007,Calder2015}, 
our setting shall be the directed number line graph or transitive tournament. \\


For $n\in \Z^+$, a \emph{transitive tournament} of size $n$,  
consists of a set $V$ of $n$ labelled elements, here taken to be $\{1,\ldots,n\}$, 
together with a directed edge relation 
$E \subset V\times V$, given by $E=\{(i,j)\in V\times V| i<j\}$. Transitive tournaments may be considered 
as finite subgraphs of the \emph{number line} $(\Z,<)$, i.e. the infinite directed acyclic graph given by $\left(\Z, \{(i,j)\in \Z^2|i<j\}\right)$.
Transitive tournaments are also a subclass of \emph{tournaments}, where for every pair of vertices $i,j \in V$, 
exactly one of the directed edges $(i,j)$, $(j,i)$, must be in $E$.
Transitive tournaments and their subgraphs have been used to model problems in several areas, e.g.~\cite{cfcb8e03d96a4bfbb5bd9eff03d36585,doi:10.1002/cpa.3160470307}.

\paragraph*{Notation}
Let $\calB(p)$ denote the Bernoulli distribution with parameter $p$,   
i.e. the distribution of a variable which is $1$ with probability $p$, and zero with probability $1-p$.
For a given assignment of edge weights to the number line, for vertices $i<j$, $w(i,j)$ denotes the weight of the edge $(i,j)$, and $w[i,j]$ denotes 
the weight of the maximum-weighted path from vertex $i$ to vertex $j$. Let $X_n$ denote $w[1,n]$.

\paragraph*{Previous Work} 
Besides the extensive literature on last passage percolation problems, for which we refer the reader to e.g.~\cite{Mart2006}, 
much research has also been devoted to problems on randomly weighted graphs. 
Walkup~\cite{doi:10.1137/0208036} and later Aldous~\cite{doi:10.1002/rsa.1015} studied minimum weight perfect matchings in 
randomly weighted complete bipartite graphs, with Aldous 
showing that the expected value converges to $\zeta(2) = \sum_{k\in \N} k^{-2} = \pi^2/6$. Frieze~\cite{FRIEZE198547,10.2307/30035650}, Karp~\cite{doi:10.1137/0208045} 
and several others also considered problems on weighted random graphs, finding limiting constants in several cases.
Closer to our problem, Denisov, Foss and Konstantopoulos~\cite{denisov2012} 
considered random subgraphs of the number line, with constant weights. Foss, Martin, and Schmidt~\cite{foss2014} considered 
maximum weight paths on random subgraphs of the number line, for general edge weight distributions. Recently, 
Yuster~\cite{DBLP:journals/dam/Yuster17} considered maximum weight paths on randomly weighted 
tournaments, for Bernoulli and uniform-distributed weights on the edges. These last two mentioned results are described in more detail below. \\

Yuster~\cite{DBLP:journals/dam/Yuster17} proved bounds on the expected weight of the maximum weighted path, for i.i.d. edge weights having 
the $\calB(1/2)$ distribution and the uniform distribution on $[0,1]$. For small $n$, he computed $\Ex{X_n}$ using a program, and gave the 
following table of values.

\begin{table}[h!]
\begin{center}
\begin{tabular}{ |c||c|c|c|c|c|c|c| }
\hline
$n$ & $3$ & $4$ & $5$ & $6$ & $7$ & $8$ \\
\hline
$\Ex{X_n}$ & $\frac{9}{8}$ & $\frac{111}{64}$ & $\frac{2399}{1024}$ & $\frac{96735}{32768}$ & $\frac{7468479}{2097152}$ & $\frac{1119481727}{26843456}$ \\
\hline
\end{tabular}
\caption{Values of $\Ex{X_n}$ for small $n$.}
\label{t:yuster-val}
\end{center}
\end{table}

Given $p\in [0,1]$, define the limit 
\begin{eqnarray}
   \beta_{tr}(p) := \lim_{n\to \infty} \frac{\Ex{X_n}}{n-1}.  \label{eqn:defn-beta}
\end{eqnarray}
For transitive tournaments with $\calB(1/2)$ random weights on the edges, Yuster showed the following.
    \begin{eqnarray}
        0.595 &\leq& \beta_{tr}(1/2) \;\;\leq\;\; 0.614. \label{eqn:yuster-bdds}
    \end{eqnarray}
These were proved using a partial recurrence relation for the upper bound, and a combinatorial partitioning argument for the lower bound.
He also asked the question of finding an exact formula for the expected weight in $k$-vertex transitive tournaments, for both the above-mentioned distributions, 
describing it as \emph{a challenging problem even for small $k$}. \\

Foss, Martin, and Schmidt~\cite{foss2014} on the other hand, approached the problem from the point of view of last passage percolation. Using ideas from renewal theory, 
they showed that the last passage percolation problem in this case has a regenerative structure, which could be used to prove very general limit laws, scaling laws,
and asymptotic distributions for i.i.d. non-negative weight distributions on the edges of random 
subgraphs of the number line. In particular for weight distributions having finite variance and third moment,
they gave a strong law of large numbers and a functional central limit theorem. These can be summarized as below.

\begin{theorem}[{~\cite[Theorems 2.1,2.4]{foss2014}}]
\label{thm:foss-etal-slln-fclt}
    For edge weight distributions having finite variance and third moment, there exist finite constants $C,c>0$, depending only on 
    the edge weight distribution, such that the following hold true:
   \begin{itemize}
      \item[(i)] $\frac{w[0,n]}{n} \to C,$ almost surely, as $n\to\infty$. 
      \item[(ii)] $\frac{w[0,n]^+}{n} \to C,$ in $\mathcal{L}_1$, as $n\to\infty$.
      \item[(iii)] Define, for every $t\geq 0$, $W(t) := \frac{w[0,nt]-Cnt}{c\sqrt{n}}$.
                   Then as $n\to \infty$, $W(t)$ converges in distribution to a standard Brownian motion.
   \end{itemize}
\end{theorem}
We observe that though bounds on the constants $C$ and $c$ can be obtained from the proofs in~\cite{foss2014}, they 
depend in general on probabilities of collections of infinitely many events, which are themselves defined using maximum-weight paths. 
In general, as with Yuster's methods, it seems very hard to obtain precise values of the scaling limits using their techniques, without 
knowing beforehand the entire probability distribution of the maximum weight path. 



\paragraph*{Our Contribution} 
We consider the last passage percolation problem on transitive tournaments on $n$ vertices, taken as 
subgraphs of the number line, on the set of nodes $\{1,\ldots,n\}$, for edge weights having the Bernoulli 
distribution with parameter $p\in (0,1)$. 
Our approach involves utilizing the combinatorial structure of the problem, as well as the discrete and binary nature
of the Bernoulli distribution, to obtain a tractable recurrence relation for the maximum weight of a path, from which,
a generating function for the expected maximum weight can be obtained. Further, since the recurrence 
relation actually applies to the random variable itself, it is possible to obtain a bivariate generating function which is 
a generating function for the \emph{probability generating function} (or Mellin transform) of our random variable.  \\

Next, using techniques from Analytic Combinatorics~\cite{DBLP:books/daglib/0023751}, we show that the generating function gives not only 
an exact formula for the expectation, but also the limiting constant $\beta_{tr}$. This bivariate generating function 
can now be utilized to obtain the scaling constants in the functional central limit theorem of Foss et al.~\cite{foss2014}, 
for the $\calB(p)$ edge weight distribution. Lastly, we briefly discuss how analytic combinatorics techniques can be used 
to obtain pointwise central limit theorems directly, (for suitably translated and scaled versions of our random variable), and 
how these can be extended to give a alternate proof of the functional central limit theorem of Foss et al. \\

Our first result therefore, is a generating function for $\Ex{X_n}$, the expected weight of a maximum-weighted path from $1$ to $n$.  \\

For $n\in \N$, let $g(n) = 1 + \Ex{X_n}$. Let $G_p(x) := \sum_{n\geq 0} g(n)x^n$.
Define $A_p(x) := \sum_{n\geq 0} (1-p)^{{n\choose 2}}x^n$, and $B_p(x) := \sum_{n\geq 0} (1-p)^{{n+1\choose 2}}x^n = A_p((1-p)x)$.

\begin{theorem}[Generating function] \label{thm:gen-func}
Given $p\in (0,1)$, and a transitive tournament $T$ on $n$ nodes with independently random weights on the edges having 
distribution $\calB(p)$, then with $B_p(x)$ and $G_p(x)$ as defined above,
	\begin{eqnarray*}
		G_p(x) &=& 1 + \frac{x}{(1-x)^2 B_p(x)}. \label{eqn:gen-func}
	\end{eqnarray*}
\end{theorem}

Theorem~\ref{thm:gen-func} follows from a recurrence relation on certain partial maximum-weighted paths, which we prove 
(see Lemma~\ref{l:rec-rel-bern-expect}). More generally, we also derive a generating function for the probability generating function of the 
maximum weighted path from $1$ to $n$. Let $Y_n := t^{X_n}$, where $t$ is a formal indeterminate variable. 
Define $Z(x,t):= \sum_{n\geq 0} \Ex{Y_n}x^n$. 
\begin{theorem} \label{thm:prob-gen-func}
	The probability generating function $\Ex{Y_n}$ of the maximum weight path $w[1,n]$, satisfies
        \begin{eqnarray}
	   \Ex{Y_n} &=& \sum_{i=1}^{n-1} t\cdot[(1-p)^{i\choose 2}-(1-p)^{{i+1\choose 2}}]\cdot \Ex{Y_{n-i}} + (1-p)^{{n\choose 2}}. \label{eqn:rec-prob-gen-func}
        \end{eqnarray}
	Therefore, 
	\[ Z(x,t) = 1 + \frac{xB_p(x)}{1-t[A_p(x)-B_p(x)]}.\]
\end{theorem}

Next, we give a combinatorial expression from the generating function $G_p(x)$. 
Define $H_p(x) := \frac{1}{B_p(x)}$. 
\begin{corollary}
	\label{cor:gn-expr}

	\begin{enumerate}
		\item[(a)]
	For $n\geq 1$, the coefficient of $x^n$ in $\frac{1}{B_p(x)}$ is 
			\[[x^n]\pth{\frac{1}{B_p(x)}} \;\;=\;\; \sum_{j=1}^{n} (-1)^j\sum_{\bfa\in \calC_{n,j}} (1-p)^{\sum_{i=1}^j{a_i+1\choose 2}}.\]
		\item[(b)]
			$g(n)$ is given by the following expressions: 
                        \begin{eqnarray}
			    g(n) &=& \sum_{m=0}^{n-1} (n-m)h_m \;\;=\;\; \sum_{m=0}^{n-1} (n-m)\sum_{\bfa\in \calC_{m}}(-1)^{l(\bfa)} (1-p)^{\sum_{i=1}^{l(\bfa)}{a_i+1\choose 2}}. 
                                     \label{eqn:gn-expr-hn-bn-cmk}\\
			    g(n) &=& \sum_{m=0}^{n-1} \sum_{j=0}^{m} \sum_{k=0}^{j} (-1)^k\sum_{\bfa\in \calC_{j,k}} (1-p)^{\sum_{i=1}^k{a_i+1\choose 2}}.
                                     \label{eqn:gn-expr-bn-cm}
                        \end{eqnarray}
	\end{enumerate}
\end{corollary}

We then investigate the limit $\beta_{tr}(p)$ using techniques from analytic combinatorics. The exact value of $\beta_{tr}(p)$ was asked by 
Yuster~\cite{DBLP:journals/dam/Yuster17} for $p=1/2$. 
From Theorem~\ref{thm:gen-func}, $G_p(x)$ is a meromorphic function for $x\in \C$, i.e. it is the ratio of two power series which converge everywhere in $\C$. 
Complex analysis then allows us to obtain an exact expression for $\beta_{tr}(p)$, for any $p\in (0,1)$.
Further, from the results of Foss et al.~\cite{foss2014} and Yuster~\cite{DBLP:journals/dam/Yuster17}, it is easy to observe a connection between 
$\beta_{tr}(p)$ and the strong law of large numbers for the expected value of the heaviest path in an acyclic tournament on $n$ nodes.
These results are stated below.


\begin{theorem}
\label{thm:beta-tr-formula}
Let $p\in (0,1)$. Then for the number line $(\Z,<)$, with random independent weights on the edges, 
having distribution $\calB(p)$, it holds that 
\begin{eqnarray*}
   \frac{X_n}{n-1}  &\to& \beta_{tr}(p), \mbox{ almost surely, and } \\
   \frac{X_n^+}{n-1} &\to& \beta_{tr}(p), \mbox{ in } \mathcal{L}_1,
\end{eqnarray*}
where the constant $\beta_{tr}(p)$ is given by, 
   \[ \beta_{tr}(p) = B_p(1)^{-1} = \pth{\sum_{n\geq 1}(1-p)^{n\choose 2}}^{-1} .\]
\end{theorem}

That is, the normalised weight $\frac{w[0,n]}{n}$ of the heaviest path 
converges almost surely and in $\mathcal{L}_1$ to $\beta_{tr}(p)$. In particular, we get the 
following improvement on Yuster's bounds.

\begin{corollary}
\label{cor:beta-tr-formula}
   For $p=1/2$, we have 
   \[\beta_{tr}(1/2) = \pth{\sum_{n\geq 1}2^{-{n\choose 2}}}^{-1} = 0.60914971106\ldots.\]
\end{corollary}


Next, using the generating function for the Mellin transform of the maximum-weighted path, 
we obtain the precise constants in the functional central limit theorem of Foss et al. 
Let $B_p'(x)$ and $B_p''(x)$ denote the first and second derivatives of $B_p(x)$ with respect to $x$. 
Differentiating the power series $B_p(x)$, it is easy to observe that $B_p'(x)$ and $B_p''(x)$ are 
bounded everywhere for $x\in \C$.

\begin{theorem}
\label{thm:fclt}
   For all $t\geq 0$, let $W_n(t) := \frac{X_{nt}-\beta_{tr}(p)(nt-1)}{\sigma_w \sqrt{n-1}}$, 
where 
   \[\sigma_w := B_p(1)^{-1}\pth{1+\frac{6B_p'(1)}{B_p(1)}-B_p(1)}^{1/2}.\] 
Then as $n\to \infty$, $W_n(t)$ converges to $W(t)$, where $W(t)$ is a standard Brownian motion.
\end{theorem}

To conclude, we discuss some possible further questions and extensions in Section~\ref{sec:concl}.

%
%
%
%
%

\section{Expected maximum weight} 
\label{sec:main-proof}

In this section we give the proofs of theorems ~\ref{thm:gen-func} and 
~\ref{thm:prob-gen-func}. The 
main idea is to prove a recursive formula for $Z_n = \Ex{t^{X_n}}$.
For integers $a<b$, $a,b\in \Z$, 
let $T[a,b]$ denote the subgraph of $\Z$ in the segment $[a,b]$, i.e. the 
directed acyclic graph with vertex set $\{a,\ldots,b\}$ and 
directed edges $\{(j,k) \in E \;|\; a\leq j < k\leq b\}$. 
The central idea is as follows. First, it is easy to observe that 
for any $n\in \N$, the distribution of $X_n$ is translation-invariant.
Further, the distribution of $w[a,b]$ in a given interval $[a,b]$, depends only 
on the arcs within the interval. \\

Thus, if the heaviest path in $[1,n]$ is known to pass through a node $j\in [n]$, 
then the distribution of the segment of the path that lies in $[j,n]$, is the 
same as that of $w[1,n-j+1]$. Furthermore, the range of the random weights is $\{0,1\}$, 
which means that there is a unique node $i \in [1,n]$, at which the weight along a maximum-weight path 
starting from node $1$, increases from  zero to $1$. Also, if we choose a maximum-weight path which has 
the minimum index for this node, say $i$, then \emph{all} 
the weights in $T[1,i-1]$ must be identically zero, and the weights of the edges going from 
$T[1,i-1]$ to $T[i,n]$ do not matter, as such an edge can increase the weight of a path at 
most from zero to $1$. \\

For the formal proof, 
we first observe that the distribution of the maximum-weight path in any interval is translation-invariant.
Given two random variables $X,Y$ we use the notation $X \overset{\calD}{=} Y$ to mean that $X$ and $Y$ are identically distributed.

\begin{lemma} \label{l:trans-invar-distn}
   For any $x,y,j\in [n]$, such that $1\leq x\leq y\leq y+j\leq n$, we have that 
   \begin{equation} \label{eqn:trans-invar-distn}
      w[x,y] \;\;\overset{\calD}{=}\;\; w[x+j,y+j].
   \end{equation}
\end{lemma}

\begin{proof} 
   Observe that the distribution of the heaviest path from $x$ to $y$ depends only on the edges in $T[x,y]$, 
   that is, the set of edges in $T[x,y]$. 
   Similarly, $w[x+j,y+j]$ depends only on the set of edges 
   $T(x+j,y+j)$. Since $T[x,y]$ is isomorphic to $T[x+j,y+j]$, the distributions are identical.
\end{proof}

Let $f(n)$ denote $\Ex{X_n}$. The next lemma is our main combinatorial result, on which all our 
subsequent theorems are based.
\begin{lemma} \label{l:rec-rel-bern-expect}
   The function $f(n)$ satisfies the recurrence
\begin{equation} \label{eqn:rec-rel-bern-expect}
   f(n) = \sum_{i=2}^n (1-(1-p)^{i-1})(1-p)^{{i-1\choose 2}}(f(n-i+1)+1),
\end{equation}
with $f(0)=f(1) = 0$.
\end{lemma}

\begin{remark}
It is easy to verify that Lemma~\ref{l:rec-rel-bern-expect} gives the values computed in Yuster's table~\ref{t:yuster-val},
(see~\cite[Table 1]{DBLP:journals/dam/Yuster17}).
\end{remark}

\begin{proof}
Let $E_i,\; 2\leq i\leq n+1$ denote the event that the $i-1$ is the maximum integer such that the induced linear order $T[1,i-1]$ on 
vertices $\{v_1,\ldots,v_{i-1}\}$ has weight zero, i.e. $(i)$ $w[1,i-1] = 0$, and $(ii)$ there exists at least one edge having 
weight $1$ in the set of arcs $\{(v_j,v_i)\; :\; 1\leq j < i\}$. (For $i = n+1$, only condition $(i)$ is applicable).  \\

Clearly the events $(E_i)_{i=2}^{n+1}$,
are mutually exclusive, since $E_i$ occurs if and only if $(i)$ the maximum-weighted path from $v_1$ to $v_{i-1}$ has zero weight, and
$(ii)$ the maximum weighted path from $v_1$ to $v_i$ has weight $1$. $(i)$ implies $E_j$ doesn't occur for all $j<i$, and $(ii)$ 
implies $E_j$ doesn't occur for all $j>i$.
The event $E_{n+1}$ implies the maximum weight path has zero weight, so its contribution to the expected weight of the heaviest path,
is zero. \\ 

We shall prove that in the case where the heaviest path has non-zero weight, the events $E_i: 2\leq i\leq n$ are in fact exhaustive, 
i.e. \emph{exactly} one of them always occurs. 
This will allow us to compute a recurrence relation for $f(n)$, by conditioning on the set of events $\{E_i: 2\leq i \leq n\}$.  \\

For a given path $P = (v_0,\ldots, 
v_i,\ldots, v_m)$, let $P[v_0,v_i]$ denote the subsequence of vertices in $P$, from $v_0$ to $v_i$. 
For a maximum weight 
path $P$ in $T[1,n]$, let $(v_r,v_s)$ be the first edge in $P$ having weight $1$, i.e. $w(r,s) =1$. 
Choose $P$ having the minimum index of the end-point $v_s$. 
We claim that $w[1,s-1] = 0$, and $w[1,s]=1$, i.e. the event $E_s$ occurs. 
First, observe that $w[1,s]\geq 1$, since the edge $(v_r,v_s) \in T[1,s]$ and $w(v_r,v_s)=1$.
Next, suppose $w[1,s-1]\geq 1$, then a maximum-weighted 
path $P'$ from $v_1$ to $v_{s-1}$ has weight $w[1,s-1] \geq 1$. Now extending $P'$ using the edge $(v_{s-1},v_s)$ (possibly zero-weighted) and attaching 
the rest of the path $P$ from $v_s$ to $v_n$, i.e. $P \leftarrow \left(P \setminus P[1,s]\right) \cup P'[1,s-1] \cup \{(v_{s-1},v_s)\}$, gives a path 
of weight at least $w(P)$ and having a non-zero edge with index less than $s$, which contadicts our choice of $P$. 

Now this set of mutually exclusive and exhaustive events can be used to get a recurrence relation. 
Suppose $E_i$ occurs. By the definition of $E_i$, 
the maximum weight of any path in $T[1,n]$, conditioned on $E_i$, is $1 + (w[i,n])$, that is, 
\begin{eqnarray}
   w[1,n]~|~E_i &\overset{\calD}{=}& 1 + w[i,n]~|~E_i. \label{eqn:condn-ei-maxwt}
\end{eqnarray}
Since the event $E_i$ depends only on the edges 
in $T[1,i]$, therefore the distribution of $w[i,n]|E_i$, which depends only on the arcs in $T[i,n]$,
	is the same as the (unconditional) distribution of $w[i,n]$, which by Lemma~\ref{l:trans-invar-distn} is the same 
	as the distribution of $w[1,n-i+1]$ in the transitive tournament $T[1,n-i+1]$, or, 
\[ w[i,n]~|~E_i \;\;\overset{\calD}{=}\;\; w[i,n] \;\;\overset{\calD}{=}\;\; w[1,n-i+1].\]
Taking expectations conditioned on $E_i$, the maximum weight in $T[i,n]$ under the event $E_i$ is $f(n-i+1)$.
Now, taking expectations in~\eqref{eqn:condn-ei-maxwt} gives 
\[ \Ex{w[1,n]~|~E_i} = 1 + f(n-i+1).\]
Finally, observe that $\Prob{E_i}= (1-(1-p)^{i-1})(1-p)^{{i-1\choose 2}}$.
Thus we have
\begin{eqnarray*}
   f(n) &=& \sum_{i=2}^n \Prob{E_i}\cdot (w[1,n]|E_i) \\
        &=& \sum_{i=2}^n \Prob{E_i}\cdot (1+f(n-i+1)) \\
        &=& \sum_{i=2}^n (1-(1-p)^{i-1})(1-p)^{{i-1\choose 2}}(1+f(n-i+1)). 
\end{eqnarray*}
\end{proof}

\subsection{Generating Functions}
\label{subsec:gen-func}

Recall that $A_p(x) := \sum_{n\geq 0} (1-p)^{{n\choose 2}}x^n = 1 + x + (1-p)x^2+(1-p)^3x^3 + \ldots$.
\begin{lemma} \label{l:ax-aqx-eqn}
	With $x$ as a formal variable, the following identity holds
	\[ 1 + xB_p(x) = A_p(x).\]
\end{lemma}

\begin{proof}
	We have that $A_p(x) = \sum_{n\geq 0} (1-p)^{{n\choose 2}}x^n$, and 
	\begin{eqnarray*}
	  B_p(x) &= & \sum_{n\geq 0} (1-p)^{{n\choose 2}}\cdot (1-p)^{n}x^n \\
	            &= & \sum_{n\geq 0} (1-p)^{{n\choose 2}+n}x^n \;\;=\;\;  \sum_{n\geq 0} (1-p)^{{n+1\choose 2}}x^n.
	\end{eqnarray*}
	Thus 
	\begin{eqnarray*}
		B_p(x) &=& \sum_{n\geq 0} (1-p)^{{n+1\choose 2}} x^n 
			   \;\;=\;\; \frac{1}{x}\left(\sum_{n\geq 0} (1-p)^{{n+1\choose 2}} x^{n+1}\right) \\
		       &=& \frac{1}{x}\left(A_p(x)-1\right). 
	\end{eqnarray*}
	Now multiplying both sides by $x$ and adding $1$ to both sides, gives the lemma.
\end{proof}

Let $g(n):= 1+f(n)$ denote the function with value $1$ more than $f(n)$.
Define $g(0)=g(1)=1$. We shall give a generating function $G_p(x) := \sum_{n\geq 0} g(n)x^n$. 

\begin{proof}[Proof of Theorem~\ref{thm:gen-func}]
	First, rewrite the recurrence in Lemma~\ref{l:rec-rel-bern-expect} as 
	\begin{eqnarray}
		f(n) &=& \sum_{i=1}^{n-1} ((1-p)^{{i\choose 2}}-(1-p)^{{i+1\choose 2}})\cdot(f(n-i)+1), \mbox{ or,} 
		         \notag \\
		g(n) &=& 1+ \sum_{i=1}^{n-1} \left[(1-p)^{{i\choose 2}}-(1-p)^{{i+1\choose 2}}\right] g(n-i). 
			 \label{eqn:rec-gn}
	\end{eqnarray}
	Now, multiplying both sides of ~\eqref{eqn:rec-gn} by $x^n$, where $x$ is an indeterminate, for all $n = 2,3,\ldots$, and summing up 
	the equations, we get 
	\begin{eqnarray}
		\sum_{n\geq 2} g(n)x^n &=& \sum_{n\geq 2} x^n + \sum_{n\geq 2}\sum_{i=1}^{n-1} \left[(1-p)^{{i\choose 2}}-(1-p)^{{i+1\choose 2}}\right]g(n-i), \mbox{ or,}
		\notag \\
		G_p(x)-1-x &=& \frac{1}{1-x}-1-x  + \sum_{n\geq 2}\sum_{i=1}^{n-1} \left[(1-p)^{{i\choose 2}}-(1-p)^{{i+1\choose 2}}\right]g(n-i), \mbox{ or,}
		\notag \\
		G_p(x)     &=& \frac{1}{1-x}  + \sum_{n\geq 2}\sum_{i=1}^{n-1} a_ig(n-i)x^n-b_ig(n-i)x^n, 
		\label{eqn:gen-func-first}
	\end{eqnarray}
	where $a_i = (1-p)^{{i\choose 2}}$ and $b_i = (1-p)^{{i+1 \choose 2}}$.
	Observe that we have $A_p(x) = \sum_{n\geq 0}a_nx^n$ and $B_p(x)= \sum_{n\geq 0} b_nx^n = A_p((1-p)x)$. 
	Comparing the coefficients of $x^n$ on both sides, we observe that 
	\begin{eqnarray*}
		\sum_{n\geq 2}\sum_{i=1}^{n-1} a_ig(n-i)x^n &=& \left(a_1x+\ldots + a_{n-1}x^{n-1}+\ldots\right)\cdot \\
		                                            & & \;\; \left(g(1)x+\ldots + g(n-1)x^{n-1}+\ldots\right) \\
		\notag \\
							    &=& \left(A_p(x)-1\right)\cdot\left(G_p(x)-1\right).
		\label{eqn:gen-func-prod-ax-gx}
	\end{eqnarray*}
	Similarly, we have that 
	\begin{eqnarray*}
		\sum_{n\geq 2}\sum_{i=1}^{n-1} b_ig(n-i)x^n &=& \left(B_p(x)-1\right)\cdot\left(G_p(x)-1\right).
		\label{eqn:gen-func-prod-bx-gx}
	\end{eqnarray*}
	Now substituting ~\eqref{eqn:gen-func-prod-ax-gx} and~\eqref{eqn:gen-func-prod-bx-gx} in ~\eqref{eqn:gen-func-first}, we get 
	\begin{eqnarray*}
		G_p(x)     &=& \frac{1}{1-x}  + (A_p(x)-1)\cdot(G_p(x)-1)-(B_p(x)-1)(G_p(x)-1) -1. 
	\end{eqnarray*}
	Simplifying, 
	\begin{eqnarray*}
		G_p(x)\left(1-A_p(x)+B_p(x)\right)-\left(B_p(x)-A_p(x)\right)     &=& \frac{x}{1-x},  \mbox{ or,} 
		\notag \\
		(G_p(x)-1)\cdot(1-A_p(x)+B_p(x))     &=& \frac{x}{1-x},   
		\notag 
	\end{eqnarray*}
	This gives 
	\begin{eqnarray*}
		G_p(x)     &=& 1 + \frac{x}{(1-x)\cdot(1-A_p(x)+B_p(x))}. 
		\label{eqn:gx-ax-aqx-fin}
	\end{eqnarray*}
	Using Lemma~\ref{l:ax-aqx-eqn}, we get that $1-A_p(x)=-xB_p(x)$. Substituting in the denominator of 
	the second term in the RHS of ~\eqref{eqn:gx-ax-aqx-fin}, we obtain 
	\[G_p(x) = 1 + \frac{x}{(1-x)^2B_p(x)}.\]
\end{proof}

\subsection{Generating Function for Mellin Transform}

\begin{proof}[Theorem~\ref{thm:prob-gen-func}]
	The key observation is that the recurrence between the conditional expectations obtained in Lemma~\ref{l:rec-rel-bern-expect}, 
	is actually a recurrence between the corresponding conditional random variables, and 
	does not depend on taking expectations. Define the events $(E_i)_{i=2}^{n+1}$, as in the proof of Lemma~\ref{l:rec-rel-bern-expect}. The arguments 
        will be similar. A point of difference though, is that 
        the event $E_{n+1}$ cannot be ignored here, since we have $\Ex{t^{X_n}|E_{n+1}} = t^0 = 1$.
        Arguing as in the proof of Lemma~\ref{l:rec-rel-bern-expect}, we get
	\begin{eqnarray*}
		\Ex{Y_n} &=& \Ex{t^{X_n}} \;\;=\;\; \sum_{i=2}^{n+1} \Prob{E_i}\cdot(\Ex{t^{X_n}|E_i}) \\
		         &=& \sum_{i=2}^{n+1} \Prob{E_i}\cdot(\Ex{t^{w[i,n]+1}|E_i})  \\
			 &=& \sum_{i=2}^n \Prob{E_i}\cdot t\cdot\Ex{t^{w[1,n-i+1]}} + \Prob{E_{n+1}} ,  
	\end{eqnarray*}
	where the last step follows from the fact that for $i \in \{2,\ldots,n\}$, we have $\Ex{t^{w[i,n]+1}} = t\Ex{t^{w[i,n]}}$, 
        and by translational invariance (Lemma~\ref{l:trans-invar-distn}), 
	$\Ex{t^{w[i,n]}} = \Ex{t^{w[1,n-i+1]}}$. 

        Now observing that $\Prob{E_i} = \pth{1-(1-p)^{i-1}}(1-p)^{{i-1\choose 2}}$ for $i=2,\ldots,n$, and 
        $\Prob{E_{n+1}} = (1-p)^{{n\choose 2}}$ gives the following.
	\begin{eqnarray*}
		\Ex{Y_n} &=& \sum_{i=2}^n (1-(1-p)^{i-1})(1-p)^{{i-1\choose 2}}(t\Ex{t^{w[1,n-i+1]}}) + (1-p)^{{n\choose 2}}\\
		         &=& \sum_{i=2}^n (1-(1-p)^{i-1})(1-p)^{{i-1\choose 2}}(t\Ex{Y_{n-i+1}}) + (1-p)^{{n\choose 2}} \\
		         &=& t\sum_{i=1}^{n-1} \pth{(1-p)^{{i\choose 2}}-(1-p)^{{i+1\choose 2}}}\Ex{Y_{n-i}}+ (1-p)^{{n\choose 2}}, 
	\end{eqnarray*}
        or, 
	\begin{eqnarray*}
		\Ex{Y_n}- (1-p)^{{n\choose 2}} &=& t\sum_{i=1}^{n-1} \pth{(1-p)^{{i\choose 2}}-(1-p)^{{i+1\choose 2}}}f^*(n-i,t). 
	\end{eqnarray*}
	Taking $(1-p)^{{n\choose 2}}$ to the other side completes the proof of the first part of the theorem.
	To obtain the equation for the generating function $Z(x,t)$, we proceed as in the proof of Theorem~\ref{thm:gen-func}.
        Multiplying~\eqref{eqn:rec-prob-gen-func} by $x^n$ and summing over all $n\geq 2$ gives 
	\begin{eqnarray*}
		Z(x,t) - 1 - x - \pth{A_p(x)-1-x}&=& t\big(Z(x,t)-1\big)\cdot\big(A_p(x)-1\big) - t\big(Z(x,t)-1\big)\cdot \big(B_p(x)-1\big) \\
			       &=& t\big(Z(x,t)-1\big)\cdot\big(A_p(x)-B_p(x)\big) \\ 
			       &=& t\big(Z(x,t)-1\big)\cdot\big(A_p(x)-B_p(x)\big). 
	\end{eqnarray*}
	Rearranging the terms gives us
	\[ Z(x,t) = 1 + \frac{xB_p(x)}{1-t\big(A_p(x)-B_p(x)\big)},\]
	which completes the proof of the second statement of the theorem.
\end{proof}

\subsection{Combinatorial Expressions for $\Ex{X_n}$}
\label{subsec:comb-expr-avg-wt}
We now give an expression for the expected maximum weight of a path as a function of $n$.
Given a positive integer $n\in \N$, define an \emph{$i$-composition} of $n \in \N$ to be any $i$-tuple of positive integers, whose sum 
is $n$, i.e. $\bfa = (a_1,\ldots,a_i):\;\; \sum_{j=1}^i a_j = n$; $a_j \in \N$ for all $j=1,\ldots, i$. Let $\calC_{n,i}$ denote the set of all
$i$-compositions of $n$. Let $\calC_n$ denote the set of all possible compositions of $n$, i.e. $\calC_n = \bigcup_{i\in [n]} \calC_{n,i}$. Given 
a composition $\bfa \in \calC_n$, let $l(\bfa)$ denote the \emph{length} of $\bfa$, i.e. $l(\bfa) = i : \bfa  \in \calC_{n,i}$.

In the rest of the paper, we adopt the following convention: given a power series $S(x) = \sum_{n\geq 0} s(n)x^n$ and a fixed $i\in \N$, the 
coefficient of $x^i$ in $S(x)$ will be denoted by $[x^i]S(x)$. \\

\begin{proof}[Corollary~\ref{cor:gn-expr}]
	\begin{enumerate}
		\item[(a)]
	         First, consider the expression $\frac{1}{B_p(x)}$. Note that for any $p\in (0,1)$ and any $x\in \C$, the function $B_p(x)$ is analytic, since the sequence 
	         $((1-p)^{{j+1\choose 2}}x^j)_{j=0}^{\infty}$ is uniformly convergent. Now writing $B_p(x) := 1 + C(x)$, where $C(x):= \sum_{j=1}^{\infty} (1-p)^{{j+1\choose 2}}x^j$, 
	         and expanding the fraction $\frac{1}{1+C(x)}$, we get 
	         \[\frac{1}{B_p(x)} = \frac{1}{1+C(x)} = \sum_{j=0}^{\infty} (-1)^j(C(x))^j.\]
	         We therefore have 
	         \[[x^n]\pth{\frac{1}{B_p(x)}} = \sum_{j=0}^n [x^n]\pth{(-C(x))^j},\]
	         i.e., the coefficient of $x^n$ in $\frac{1}{B_p(x)}$ is the sum of the coefficients of $x^n$ in $(-1)^j(C(x))^j$, over all non-negative integers $j$.
			Any term in $[x^n](-C(x))^j$ is of the form $(-1)^j(1-p)^{\sum_{i=1}^{j}{a_i+1\choose 2}}$, where $\mathbf{a}\in \calC_{n,j}$
			is a $j$-composition of $n$, and each such $j$-composition contributes a term to $[x^n](-C(x))^j$.
	         Therefore, summing over all possible $j$ and all $j$-tuples in $\calC_{n,j}$, we get the claimed expression for $[x^n]\frac{1}{B_p(x)}$.
	         This proves the first statement of the corollary. 
	
                \item[(b)]
			To prove ~\eqref{eqn:gn-expr-hn-bn-cmk}, we expand the expression $\frac{x}{(1-x)^2}$ to get 
			$\sum_{j=1}^{\infty} jx^j$. Now comparing coefficients gives us that 
			for $n\geq 1$,
			\begin{eqnarray*}
				g(n) &=& [x^n]G_p(x) \;\;=\;\; [x^n]\pth{\frac{x}{(1-x)^2B_p(x)}}  \\
				     &=& [x^n]\pth{\sum_{j\geq 1}jx^j}\cdot H_p(x) \\
				     &=& \sum_{j=1}^nj[x^{n-j}]H_p(x)  \;\;=\;\; \sum_{j=1}^{n} jh_{n-j}\\
				     &=& \sum_{j=0}^{n-1}(n-j)h_j, 
			\end{eqnarray*}
			where in the last step we substituted the variable $j$ in the previous expression by $n-j$. This gives the 
                        first equality in~\eqref{eqn:gn-expr-hn-bn-cmk}.
			The second equality in~\eqref{eqn:gn-expr-hn-bn-cmk} follows by rewriting the above expression as a triangular sum. We get 
			\begin{eqnarray*}
				g(n) &=& \sum_{m=0}^{n-1}\sum_{j=0}^{m}[x^{j}]\pth{\frac{1}{B_p(x)}}. 
			\end{eqnarray*}
			Substituting the value of $[x^j]\pth{\frac{1}{B_p(x)}}$ from the proof of $(a)$, now completes our proof. 
	\end{enumerate}
\end{proof}


\section{Limit Laws}

    We shall now investigate the asymptotics of the expected weight of the heaviest path in $T[1,n]$, when scaled by the number of nodes, that is, 
    the value of $\beta_{tr} = \lim_{n\to\infty} \frac{X_n}{n-1}$. 
    The main result of this section will be a bound on the approximation error of $\beta_{tr}$, when the generating function 
    $G_p(x)$ is approximated by simpler rational generating functions.
    We first present some preparatory lemmas. 
    Recall from Section~\ref{subsec:comb-expr-avg-wt} that the coefficient of $x^i$ in a power series $S(x)$ in $x$, is denoted 
    by $[x^i]S(x)$.

\begin{lemma}
\label{l:dec-seq-cmplx-rts}
Let $(a_n)_{n\geq 0}$ be a decreasing positive sequence of reals, with $a_0\neq 0$. Then the power series $\mathcal{F}(z) = \sum_{j\geq 0} a_jz^k$ has 
no roots in $\{z\in \C\;:\; \|z\|\leq 1\}$.
\end{lemma}

\begin{proof}
   Let $z$ be a complex root of $\mathcal{F}(x) =0$, with $\|z\|\leq 1$. Consider $(1-z)\mathcal{F}(z)$. We have 
\begin{eqnarray*}
   (1-z)\mathcal{F}(z)-a_0 &=& \sum_{j\geq 1} (a_j-a_{j-1})z^j \\
\end{eqnarray*}
Taking absolute values and applying the triangle inequality, we get 
\begin{eqnarray*}
   \left|(1-z)\mathcal{F}(z)-a_0\right| &=&     \left|\sum_{j\geq 1} (a_j-a_{j-1})z^j\right| 
                                                \;\;\leq\;\;  \sum_{j\geq 1} (a_j-a_{j-1}) |z|^j \\
                                        &\leq&  \sum_{j\geq 1} (a_j-a_{j-1})  \;\;=\;\;     -a_0.  
\end{eqnarray*}
where the penultimate step follows since $|z|^j \leq 1$ for all $j\geq 1$.
Thus we get $|a_0| \leq -a_0$,
which is a contradiction since we assumed $a_0$ to be a positive non-zero real number.
\end{proof}

   \begin{proposition}
   \label{prop:bx-analy-geq1}
       For any $p\in (0,1]$, $B_p(x)$ and $B_p(x)^2$ are analytic, bounded functions of $x$ for all $x\in \C$. Moreover, there exists a real $r_p>1$
     such that $B_p(x)$ and $B_p(x)^2$ are non-zero for all $x$ with $\|x\|\leq r$.
   \end{proposition}

   \begin{proof}
       The function $B_p(x)$ being a power series in $x$, with coefficients $(1-p)^{{n\choose 2}}$ going to zero as $n\to \infty$, is clearly 
    analytic and bounded, for all $x\in \C$. From this it also follows that $B_p(x)^2$ is analytic and bounded for all $x\in\C$.
    The fact that $B_p(x)$ has no complex root in $\|x\|\leq 1$, follows from the observation that $\pth{(1-p)^{{n\choose 2}}}_{n\geq 0}$ 
    is a decreasing sequence of positive reals for all $p<1$, and applying Lemma~\ref{l:dec-seq-cmplx-rts}. 
   \end{proof}

   \begin{lemma}
   \label{l:merom-expan}
     For any $p\in (0,1]$, let $B(x) := B_p(x)$, $Y(x) := Y_p(x) = \frac{1}{(1-x)^2B(x)}$, and $S(x) := S_p(x)= \frac{1}{(1-x)^3B(x)^2}$.
     Then there exists a real $r = r_p >0$, such that for $x\in \C$, with $|x|\leq r$, we have  
     \[ Y(x) = c_{2}(x-1)^{-2} + c_{1}(x-1)^{-1} + Y_a(x),\] 
     where $c_{2} = \frac{1}{B(1)}$, $c_{1} = -\frac{B'(1)}{B(1)^2}$, and $Y_a(x)$ is analytic in $x$,
     and 
     \[ S(x) = d_{3}(x-1)^{-3} + d_{2}(x-1)^{-2} + d_{1}(x-1)^{-1} + S_a(x),\] 
     where $d_{3} = \frac{-1}{B(1)^2}$, $d_{2} = \frac{2B'(1)}{B(1)^3}$, 
     $d_{1} = -3\frac{B'(1)^2}{B(1)^4} + \frac{B''(1)}{B(1)^3}$, 
     and $S_a(x)$ is analytic in $x$.
   \end{lemma}

   \begin{proof}
    By Proposition~\ref{prop:bx-analy-geq1}, $Y(x)$ and $S(x)$ are 
    meromorphic functions, for all $x\in \C$, and for any fixed $p\in (0,1]$, there exists a real $r = r_p$, such that 
    for all $x$ in the disc $|x| \leq r_p$ in $\C$, both $Y(x)$ and $S(x)$ have a unique dominant pole at $x=1$. For $Y(x)$, the pole at $x=1$ 
    is of order $2$ and for $S(x)$, it is of order $3$. Thus, expanding $Y(x)$ and $S(x)$ respectively near $x=1$, we get 
    \begin{eqnarray}
       Y(x) &=&  \sum_{j\geq -2}c_{-j}(x-1)^j, \notag
    \end{eqnarray}
    and 
    \begin{eqnarray}
       S(x) &=&  \sum_{j\geq -3}d_{-j}(x-1)^j, \notag
    \end{eqnarray}
    \footnote{We name the constants $c_{-j}$ rather than $c_j$ for notational convenience in the later proofs.} 
    where the constants $c_{2}, c_{1},\ldots \in \C$ and $d_{3}, d_{2},\ldots \in \C$ can be obtained by differentiation and substitution. 
    For $Y(x)$, we get $c_{2} = (x-1)^2Y(x)|_{x=1} = \frac{1}{B(1)}$, $c_{1} = \left.\frac{\partial}{\partial x} ((x-1)^2Y(x))\right|_{x=1} = \frac{-B'(1)}{B(1)^2}$.
    Similarly, for $S(x)$, we get $d_{3} = (x-1)^3S(x)|_{x=1} = \frac{-1}{B(1)^2}$,  
    \[ d_{2} = \left.\frac{\partial }{\partial x}((x-1)^3S(x))\right|_{x=1} = \frac{2B'(1)}{B(1)^3},\]
    and 
    \[ d_{1} = \pth{\frac{1}{2}}\left.\frac{\partial^2}{\partial x^2}( (x-1)^3S(x))\right|_{x=1} = -3\frac{B'(1)^2}{B(1)^4} + \frac{B''(1)}{B(1)^3}.\]
    Taking $Y_a(x) = \sum_{j\geq 0} c_{-j}(x-1)^j$ and $S_a(x) = \sum_{j\geq 0} d_{-j}(x-1)^j$, we get the statement of the lemma.
   \end{proof}

    The following theorem from the book by Flajolet and Sedgewick~\cite{DBLP:books/daglib/0023751} will be very useful in proving our main results in 
    this section.
    \begin{theorem}[\cite{DBLP:books/daglib/0023751}[Theorem IV.10]]
    \label{thm:analy-comb-iv-10}
       Let $\mathcal{F}:\C\to \C$ be a function meromorphic at all points in $|x| \leq r$, with poles at $\beta_1,\ldots, \beta_s$. 
    If $\mathcal{F}(x)$ is analytic at $|x|= r$, and at $x=0$, then there exist $s$ polynomials $P_1(x),\ldots,P_s(x)$, such that 
    \[f_n \equiv [x^n]\mathcal{F}(x) = \sum_{j=1}^s P_j(n)\beta_j^{-n} + O(r^{-n}),\]
    where for each $j=1,\ldots, s$, the degree of the polynomial $P_j(x)$ is one less than the order of the pole $\beta_j$.
    \end{theorem}
   \begin{proof}[Theorem~\ref{thm:beta-tr-formula}]
    Our goal is to compute $\lim_{n\to\infty} \frac{w[0,n]}{n}$.
    Since the Bernoulli distribution has a finite variance and third moment for any $p\in (0,1]$, from 
    Lemma~\ref{l:trans-invar-distn} and Theorem~\ref{thm:foss-etal-slln-fclt} $(i)$ and $(ii)$, we know 
    that there exists a constant $C>0$, such that 
    \begin{eqnarray}
       \lim_{n\to \infty} \frac{w[0,n]}{n} &=& \lim_{n\to \infty} \frac{w[1,n+1]}{n} = \lim_{n\to\infty} \frac{X_{n+1}}{n} = \lim_{n\to \infty} \frac{X_n}{n-1} = C, \label{eqn:big-c-eq-beta}
    \end{eqnarray}
    almost surely, and similarly, 
    \[\lim_{n\to \infty} \frac{w[0,n]^+}{n} =  \lim_{n\to \infty} \frac{X_n^+}{n-1} = C,\] in $\mathcal{L}_1$. Taking expectations in~\eqref{eqn:big-c-eq-beta} gives 
     \[ C = \lim_{n\to \infty} \frac{\Ex{X_n}}{n-1} = \beta_{tr}(p).\]
    From Theorem~\ref{thm:gen-func}, we have that 
    \[\lim_{n\to\infty}\frac{\Ex{X_n}}{n-1} =\lim_{n\to \infty} \frac{1}{n-1}[x^{n}]G_p(x).\]
    For $n\geq 1$, we have 
    \[[x^{n}]G_p(x) = [x^{n-1}]\frac{x}{(1-x)^2B_p(x)} = [x^{n-1}]\frac{1}{(1-x)^2B_p(x)}.\]
    Now applying Theorem~\ref{thm:analy-comb-iv-10} to the function $\frac{1}{(1-x)^2B_p(x)}$ with $r=r_p>1$ from Proposition~\ref{prop:bx-analy-geq1}, and 
    $x=1$ as the only pole, having order 2, we see that there exists a polynomial $P_1(x)$ having degree $1$, such that 
    \begin{eqnarray}
       [x^{n-1}]\frac{1}{(1-x)^2B_p(x)} &=& P_1(n-1) + O(r^{-n+2}).  \label{eqn:anacomb-expan}
    \end{eqnarray}
    By Lemma~\ref{l:merom-expan}, there exists $r>1$ such that 
    \begin{eqnarray*}
       [x^{n-1}]\frac{1}{(1-x)^2B_p(x)} &=& [x^{n-1}]c_{2}(1-x)^{-2} + [x^{n-1}](-c_{1})(1-x)^{-1} + [x^{n-1}]Y_a(x), 
    \end{eqnarray*}
    where $c_{2} = \frac{1}{B_p(1)}$, $c_{1} = \frac{-B'(1)}{B_p(1)^2}$ and $Y_a(x)$ is analytic in $|x|\leq r$.
    Now applying Newton's expansion for negative binomials gives
    \begin{eqnarray}
       [x^{n-1}]\frac{1}{(1-x)^2B_p(x)} &=& c_{2}(n) - c_{1} + [x^{n-1}]Y_a(x), \label{eqn:gpx-merom-newt-expan}
    \end{eqnarray}
    Using Cauchy's formula for contour integration at $|z|=r$ and bounding the integrand (see e.g.~\cite[Proof of Theorem IV.10]{DBLP:books/daglib/0023751}), 
    we compute 
    $[x^{n-1}]Y_a(x) = \frac{1}{2\pi\iota}\int_{|z|=r} \frac{Y_a(z)}{z^{n}}dz = O(r^{-n+1})$, where $\iota$ denotes $\sqrt{-1}$.
    Comparing~\eqref{eqn:anacomb-expan} with the 
    expansion~\eqref{eqn:gpx-merom-newt-expan}, we get $P_1(x-1) = c_{2}(x) - c_{1}$,
    and 
    \[ [x^{n-1}]\frac{1}{(1-x)^2B_p(x)} = c_{2}(n) - c_{1} + O(r^{-n+1}).\]
    Thus, 
    \begin{eqnarray*}
       \lim_{n\to \infty} \frac{\Ex{X_n}}{n-1} &=& [x^{n-1}]\frac{1}{(1-x)^2B_p(x)} 
                                                  \;\;=\;\; c_{2} - \lim_{n\to\infty} \frac{c_1}{n-1} + \lim_{n\to\infty} \frac{r^{-n+1}}{n-1}  \\
                                               &=& c_2 \;\;=\;\; \frac{1}{B_p(1)},
    \end{eqnarray*}
    since $\lim_{n\to\infty}\frac{c_{1}}{n-1} = 0$, and $r>1$ implies $\lim_{n\to \infty} \frac{r^{-n+1}}{n-1} = 0$. 
    Together with~\eqref{eqn:big-c-eq-beta}, this gives that $\beta_{tr}(p) = C = B_p(1)^{-1}$, completing the proof of the theorem.
   \end{proof}

   Next, we'll see how the bivariate generating function for $X_n$ can be used to obtain precise expressions for the scaling constants
   in the functional central limit theorem for $X_n$ given by Foss et al.~\cite{foss2014}.
   For a random variable $X$, let $\Var(X)$ denote its variance.
   \begin{lemma}
   \label{l:var-xn-scale-lim}
       \begin{eqnarray}
          \lim_{n\to\infty} \frac{\Var(X_n)}{n-1} &=& B_p(1)^{-2}\pth{1-\frac{2B_p'(1)}{B_p(1)^{3}}} . \notag
       \end{eqnarray}
   \end{lemma}
 
   \begin{proof}
      Recall that $Z_p = Z_p(x,t)$ denotes the generating function for the probability generating of $X_n$.
      The variance $\Var(X)$ is given by 
      \begin{eqnarray*}
         \Var(X_n) &=& \Ex{X_n^2}-\Ex{X_n}^2 \\ 
                &=& [x^n]\pth{\frac{\partial }{\partial t} \pth{t\frac{\partial Z}{\partial t}}}_{t=1} - \Ex{X_n}^2.
      \end{eqnarray*}
      From Theorem~\ref{thm:prob-gen-func}, the probability generating function for $X_n$ is given by $Z = Z(x,t) = 1 + \frac{xB_p(x)}{1-t(A_p(x)-B_p(x))}$.
      For convenience, we shall use the shorthand $A := A_p(x)$, $B := B_p(x)$, and $D := A_p(x) - B_p(x)$. We have 
      \begin{eqnarray*}
          \Ex{X_n^2} &=& [x^n]\pth{\frac{\partial }{\partial t} \pth{t\frac{\partial Z}{\partial t}}}_{t=1} \\
                     &=& [x^n]\pth{\frac{\partial }{\partial t} \pth{t\frac{\partial }{\partial t} \pth{1 + \frac{xB}{1-tD}}}}_{t=1} \\
                     &=& [x^n]\pth{\frac{\partial }{\partial t} \pth{\frac{xBDt}{(1-tD)^2}}}_{t=1} \\
                     &=& [x^n]\pth{ \frac{xBD(1+tD)}{(1-tD)^3}}_{t=1}. 
      \end{eqnarray*}
      Substituting $t=1$ and using that $D = A-B$,  and Lemma~\ref{l:ax-aqx-eqn}, we get 
      that $1-D = 1-A + B = -xB + B = (1-x)B$, and $1+D = 1 + A - B = 2 - (1-x)B$.
      Thus,
      \begin{eqnarray}
          \Ex{X_n^2} &=& [x^n]\pth{\frac{2xBD}{(1-x)^3B^3} - \frac{x(1-x)B^2D}{(1-x)^3B^3}}  \notag \\
                     &=& [x^n]\frac{2xD}{(1-x)^3B^2} - [x^n]\frac{xD}{(1-x)^2B} \notag \\ 
                     &=& [x^n]\frac{2x(1-(1-x)B)}{(1-x)^3B^2} - [x^n]\frac{x(1-(1-x)B)}{(1-x)^2B}   \label{eqn:lem-xn-sq-3} \\
                     &=& 2[x^{n-1}]\frac{1}{(1-x)^3B^2} - 3[x^{n-1}]\frac{1}{(1-x)^2B}+[x^{n-1}]\frac{1}{1-x} . \label{eqn:lem-xn-sq-two-terms}
      \end{eqnarray}
      As in the previous proof, Proposition~\ref{prop:bx-analy-geq1} gives that there exists a real $r = r_p > 1$, such that 
      the functions $B_p(x)$ and $B_p(x)^2$ are analytic for all $|x|\leq r$, with a unique zero at $x=1$.
      Now applying Theorem~\ref{thm:analy-comb-iv-10} to the function $\frac{2}{(1-x)^3B(x)^2}-\frac{3}{(1-x)^2B(x)}+\frac{1}{1-x}$, 
      with $r$ as the radius of meromorphicity, we obtain that 
      \begin{eqnarray}
          [x^n]\frac{2x}{(1-x)^3B_p^2(x)} &=& 2[x^{n-1}]\frac{1}{(1-x)^3B_p^2(x)}  \notag\\ 
                                     &=& 2[x^{n-1}](-d_{3})(1-x)^{-3} + 2[x^{n-1}](d_{2})(1-x)^{-2}  \notag\\ 
                                     & & \;\;+\;\; 2[x^{n-1}](-d_{1})(1-x)^{-1} + 2[x^{n-1}]S_a(x), \notag
      \end{eqnarray}
      where $d_{3} = -B_p(1)^{-2}$, $d_{2} = 2B_p'(1)B_p(1)^{-3}$, $d_{1} = -3B_p'(1)^2B_p(1)^{-4}+B_p(1)^{-3}B_p''(1)$ and $S_a(x)$ is analytic in $|x|\leq r$.
      Similarly, 
      \begin{eqnarray*}
         [x^{n-1}]\frac{1}{(1-x)^2B_p(x)} &=& [x^{n-1}]c_{2}(1-x)^{-2} + [x^{n-1}](-c_{1})(1-x)^{-1} + [x^{n-1}]Y_a(x), 
      \end{eqnarray*}
      where $c_{2} = \frac{1}{B_p(1)}$, $c_{1} = \frac{-B_p'(1)}{B_p(1)^2}$ and $Y_a(x)$ is analytic in $|x|\leq r$.
      Finally, $[x^{n-1}]\frac{1}{1-x} = 1$.
      Now using Newton's expansion for negative binomials (see e.g.~\cite[Theorem IV.9]{DBLP:books/daglib/0023751}), we get 
      \begin{eqnarray}
         \Var(X_n) &=& \Ex{X_n^2} - (\Ex{X_n})^2 \notag \\
                  &=& -2d_{3}{n+1\choose 2} + 2d_{2}(n) - 2d_{1} - 3(c_{2}n-c_{1}-1 + o(1)) \notag \\ 
                  & & \;\; -(c_{2}n-c_{1}-1 + o(1))^2 \notag \\
                  &=& n^2\pth{-d_{3}-c_{2}^2}+n\pth{-d_{3}+2d_{2}-3c_{2}+2c_{2}(1+c_{1})}+ o(n). \label{eqn:var-xnsq-di-ci}
      \end{eqnarray}
      Substituting the values of $d_3$, $d_2$, $c_2$ and $c_1$ in~\eqref{eqn:var-xnsq-di-ci}, we get
      \begin{eqnarray}
          \Var(X_n)
                  &=& n^2\pth{\frac{1}{B_p(1)^{2}}-\frac{1}{B_p(1)^{2}}}+ \notag \\
                  & & \;\; n\pth{\frac{1}{B_p(1)^2}+4\frac{B_p'(1)}{B_p(1)^3}-\frac{3}{B_p(1)}+2\frac{1}{B_p(1)}\pth{1-\frac{B_p'(1)}{B_p(1)^{2}}}}+ o(n) \notag \\ 
                  &=& \frac{n}{B_p(1)^{2}}\pth{1+\frac{6B_p'(1)}{B_p(1)}-B_p(1)}+ o(n).  \label{eqn:var-scale-lim}
      \end{eqnarray}
      Dividing both sides of~\eqref{eqn:var-scale-lim} by $n-1$ and taking the limit as $n\to\infty$, we get the statement of the theorem.
      
   \end{proof}

   \begin{proof}[Theorem~\ref{thm:fclt}]
      Let $X_n = X_n(t):= w[1,nt]$. Since the Bernoulli distribution has finite variance and third moment for any choice of $p\in (0,1)$, by 
      Foss et al. (Theorem~\ref{thm:foss-etal-slln-fclt}), ~\cite{foss2014}[Theorem 2.4], we know that there exist $C,c>0$ such that $W_n(t)$
      converges to a standard Brownian motion as $n\to\infty$. Further, from Theorem~\ref{thm:gen-func}, we have that for any $t\geq 0$, 
      $\Ex{X_n(t)} \to \beta_{tr}(p)(nt-1)$ as $n\to \infty$. Therefore in order to prove the theorem, we only need to 
      find the scaling limit for the variance of $W_n(t)$. 
      For $t=1$, Theorem~\ref{thm:foss-etal-slln-fclt} gives that $W_n(1)$ converges to a standard normal distribution, i.e. has variance $1$. 
      Now comparing variances gives that the scaling 
      constant $c = \pth{\lim_{n\to\infty} \frac{\Var(X_n)}{n-1}}^{1/2}$. From Lemma~\ref{l:var-xn-scale-lim}, we get that 
      \[ \lim_{n\to\infty} \frac{\Var(X_n)}{n-1} = B_p(1)^{-2}\pth{1+\frac{6B_p'(1)}{B_p(1)}-B_p(1)}.\]
      Taking square roots therefore, completes the proof of the theorem.
   \end{proof}

\section{Conclusions}
\label{sec:concl}

      For the Bernoulli distribution, we were able to find precise values of the scaling constants for the strong law of natural numbers
      and the functional central limit theorem using techniques from enumerative and analytic combinatorics. In some sense, our 
      techniques are a simplification of the regenerative structure described by Foss et al.~\cite{foss2014} to the case of Bernoulli-distributed 
      weights. This connection can be extended further, to get a direct proof of the functional limit theorem of Foss et al.~\cite{foss2014} for 
      Bernoulli-distributed weights (without using Donsker's theorem). We give a brief outline of the proof here.\\

      First, it is straightforward to show that the random variable $W_n(t)$ in Theorem~\ref{thm:fclt} weakly converges pointwise to a normal random 
      variable with variance $t$, using singularity perturbation for meromorphic functions~\cite[Theorem IX.9]{DBLP:books/daglib/0023751} on the 
      bivariate generating function $Z(x,t)$, and verifying that the conditions for applicability are satisfied. Next, we need to prove that for every 
      $0=t_0 < t_1 < t_2 < \ldots$, the variables $W(t_{i_2})-W(t_{i_1})$, $W(t_{i_3})-W(t_{i_2})$, etc. are mutually independent. Consider the 
      variable $W(t_2)-W(t_1) = \lim_{n\to\infty}\frac{X_{nt_2}-X_{nt_1}-(Cnt_2-Cnt_1)}{\sigma_w\sqrt{n-1}}$.  

      By ``unravelling" the recursive argument in the proof of Lemma~\ref{l:rec-rel-bern-expect}, it can be shown that asymptotically the 
      distribution of $X_{nt_2} - X_{nt_1}$, converges to that of $X_{n(t_2-t_1)}$. This is because the difference 
      $D_n(t) = X_{nt_2}-X_{nt_1}-X_{n(t_2-t_1)}$ can only be due to an edge in a maximum-weight path that jumps the node $nt_1$, i.e. 
      has exactly one end-point in $[0,nt_1)$ and the other in $(nt_1,nt_2]$. Thus, $D_n(t)$ is 
      at most $1$, which disappears after scaling by $\sigma_w\sqrt{n-1}$, $n$ goes to infinity in the limit. 

      Lastly, to see that 
      the function $t\mapsto W(t)$ is continous, it suffices to use the pointwise convergence to a normal distribution with variance $t$ 
      and independence of increments proved already, together with standard concentration bounds for the normal distribution.\\

      It would be interesting 
      to know if these combinatorial techniques can be extended to other distributions, for example the uniform distribution on $[0,1]$, or 
      even the uniform discrete distribution on $\{0,\ldots,k-1\}$ for some $k\geq 2$. The 
      main obstruction is that $k$-valued distributions are not binary, and therefore, it is possible for edges that ``jump" a given point, 
      to supercede the weight of any path that passes through the point.
      Other questions of interest such as the asymptotic expectation for the general class of tournaments, remain open.

\bibliographystyle{plainurl}
\bibliography{biblio}
\end{document}